\documentclass[12pt,a4paper,reqno]{amsart}
\usepackage{amsmath}
\usepackage{amsfonts}
\usepackage{amssymb}
\usepackage{amsthm}
\usepackage{enumerate}
\usepackage{centernot}
\usepackage{fullpage}
\usepackage{mathrsfs}
\usepackage{graphicx}
\usepackage{color}

\newcommand{\upperRomannumeral}[1]{\uppercase\expandafter{\romannumeral#1}}

\theoremstyle{plain}

  \newtheorem{proposition}[]{Proposition}
  \newtheorem{lemma}[]{Lemma}
  \newtheorem{theorem}[]{Theorem}
  \newtheorem{corollary}[]{Corollary}
  
  \newtheorem{remark}[]{Remark}

\title[First Lee-Yang zero]{Thermodynamic limit of the first Lee-Yang zero}

\author{Jianping Jiang}
\address{Yau Mathematical Sciences Center, Tsinghua University, Beijing 100084, China.}
\email{jianpingjiang@tsinghua.edu.cn}
\author{Charles M. Newman}
\address{Courant Institute of Mathematical Sciences, New York University,
	251 Mercer st, New York, NY 10012, USA, \& NYU-ECNU Institute of Mathematical
	Sciences at NYU Shanghai, 3663 Zhongshan Road North, Shanghai 200062, China.}
\email{newman@cims.nyu.edu}

\begin{document}
\begin{abstract}
We complete the verification of the 1952 Yang and Lee proposal that thermodynamic singularities are exactly the limits in ${\mathbb R}$ of finite-volume singularities in ${\mathbb C}$.
For the Ising model defined on a finite $\Lambda\subset\mathbb{Z}^d$ at inverse temperature $\beta\geq0$ and external field $h$, let $\alpha_1(\Lambda,\beta)$ be the modulus of the first zero (that closest to the origin) of its partition function (in the variable $h$). We prove that $\alpha_1(\Lambda,\beta)$ decreases to $\alpha_1(\mathbb{Z}^d,\beta)$ as $\Lambda$ increases to $\mathbb{Z}^d$ where $\alpha_1(\mathbb{Z}^d,\beta)\in[0,\infty)$ is the radius of the largest disk centered at the origin in which the free energy in the thermodynamic limit is analytic. We also note that $\alpha_1(\mathbb{Z}^d,\beta)$ is strictly positive if and only if $\beta$ is strictly less than the critical inverse temperature.
\end{abstract}

\maketitle

\section{Introduction and main result}
In this paper, we finish the rigorous verification for Ising models of the program initiated by Yang and Lee \cite{YL52} and Lee and Yang \cite{LY52} that singularities of thermodynamic functions in the physical real parameter space result from the pinching of the real line in the thermodynamic limit  by singularities from the complex plane.

We consider the Ising model on $\mathbb{Z}^d$ with nearest-neighbor ferromagnetic pair interactions. The results in this paper can be extended to finite-range ferromagnetic pair interactions, but we restrict to the nearest-neighbor case for the sake of simplicity. Let $\Lambda$ be a finite subset of $\mathbb{Z}^d$. We slightly abuse the notation and also write $\Lambda$ for the graph with the set of vertices $\Lambda$ and the set of edges being all nearest-neighbor edges between vertices in $\Lambda$. The Ising model on $\Lambda$ at inverse temperature $\beta\geq0$ with free boundary conditions and external field $h\in\mathbb{R}$ is defined by the probability measure $\mathbb{P}_{\Lambda,\beta,h}$ on $\{-1,+1\}^{\Lambda}$ such that
\begin{equation}\label{eq:Isingdef}
	\mathbb{P}_{\Lambda,\beta,h}(\sigma):=\frac{\exp\left[\beta\sum_{uv\in \Lambda}\sigma_u\sigma_v+h\sum_{u\in \Lambda}\sigma_u\right]}{Z_{\Lambda,\beta,h}}, \sigma\in\{-1,+1\}^{\Lambda},
\end{equation}
where the first sum is over all nearest-neighbor edges $uv$ with both $u$ and $v$ in $\Lambda$, and $Z_{\Lambda,\beta,h}$ is the partition function. More precisely,
\begin{equation}
	Z_{\Lambda,\beta,h}:=\sum_{\sigma\in\{-1,+1\}^{\Lambda}}\exp\left[\beta\sum_{uv\in \Lambda}\sigma_u\sigma_v+h\sum_{u\in \Lambda}\sigma_u\right].
\end{equation}
A seminal result due to Lee and Yang \cite{LY52} says that all zeros of $Z_{\Lambda,\beta,h}$ (in the variable $h\in\mathbb{C}$ with $\beta\geq 0$ fixed) are purely imaginary. See \cite{Rue71,SG73,New74,LS81,NG83,BBCKK04,BBCK04} for various generalizations of the Lee-Yang result. We may assume that all the zeros are labelled as $\pm i\alpha_j(\Lambda,\beta)$ such that
\begin{equation}
	0<\alpha_1(\Lambda,\beta)\leq\alpha_2(\Lambda,\beta)\leq\dots.
\end{equation} 
In the Lee-Yang program \cite{YL52,LY52}, it was argued that only for a temperature at or below the critical one, in the thermodynamic limit, complex singularities of the free energy pinch the physical (i.e., real) domain. In this paper, we prove this conjecture rigorously.

We define
\begin{equation}\label{eq:f_Vdef}
	f_{\Lambda,\beta}(h):=\frac{\ln Z_{\Lambda,\beta,h}}{|\Lambda|},
\end{equation}
where $|\Lambda|$ is the number of vertices in $\Lambda$; except for a multiplicative constant and a minus sign, this is the \textbf{free energy in $\Lambda$} (sometimes it is referred to as the pressure in $\Lambda$~\cite{FV18}). For nice subsets of $\mathbb{Z}^d$, e.g., $B_n:=[-n,n]^d\cap\mathbb{Z}^d$, it is well-known that $f_{B_n,\beta}$ converges as $n\rightarrow \infty$. That is, there is a function $f_{\beta}:\mathbb{R}\rightarrow\mathbb{R}$ such that
\begin{equation}
	f_{\beta}(h):=\lim_{n\rightarrow\infty}f_{B_n,\beta}(h), h\in \mathbb{R}.
\end{equation}
Here $f_{\beta}$ is called the \textbf{free energy}; see, e.g., Theorem 3.6 of \cite{FV18} for a proof. Let $\beta_c(d)$ be the critical inverse temperature. It is known (see, e.g., Theorem 3.25 of \cite{FV18}) that $\beta_c(1)=\infty$ and $\beta_c(d)\in(0,\infty)$ for $d\geq 2$. Our main result is
\begin{theorem}\label{thm:1LY}
	Let $\beta\geq0$ and $\{\Lambda_n\}$ be a sequence of finite subsets of $\mathbb{Z}^d$ satisfying $\Lambda_n\subset\Lambda_{n+1}$ for each $n\in\mathbb{N}$ and $\cup_{n=1}^{\infty}\Lambda_n=\mathbb{Z}^d$. Let $\alpha_1(\Lambda_n,\beta)$ be the modulus of the first zero of $Z_{\Lambda_n,\beta,h}$. Then we have
	\begin{equation}
		\alpha_1(\Lambda_n,\beta)\downarrow \alpha_1(\mathbb{Z}^d,\beta) \text{ as }n\uparrow\infty,
	\end{equation}
where $\alpha_1(\mathbb{Z}^d,\beta)\in[0,\infty)$ is the radius of the largest disk centered at the origin where the free energy $f_{\beta}$ is analytic. Moreover, $\alpha_1(\mathbb{Z}^d,\beta)>0$ if and only if $\beta\in[0,\beta_c(d))$.
\end{theorem}
\begin{remark}
	A priori, the limit of the first Lee-Yang zero could have nothing to do with the singularities of the free energy. That is, for $\beta<\beta_c(d)$ $\alpha_1(\Lambda_n,\beta)$ could approach $0$ as $n\rightarrow\infty$, and for $\beta\geq \beta_c(d)$  the free energy could have no singularity at the limit position of $\alpha_1(\Lambda_n,\beta)$. Theorem \ref{thm:1LY} rules out both these possibilities.
\end{remark}

An immediate consequence of Theorem \ref{thm:1LY} is the following corollary.

\begin{corollary}\label{cor:zero-free}
	For $\beta\in[0,\beta_c(d))$ and any finite domain $\Lambda\subset\mathbb{Z}^d$, the partition function $Z_{\Lambda,\beta,h}$ is zero-free in the open disk centered at $h=0$ with radius $\alpha_1(\mathbb{Z}^d,\beta)$. Thus, there is a fixed zero-free neighborhood of the origin for the partition function of the Ising model for all finite domains if and only if $\beta\in[0,\beta_c(d))$.
\end{corollary}

\begin{remark}
	Corollary \ref{cor:zero-free} generalizes the zero-free region result of \cite{Rue71,Rue73} (see also \cite{PR20}) from small $\beta$ to all $\beta\in[0,\beta_c(d))$. In particular, this implies that Corollary 1 of \cite{CJN21} extends to all $\beta\in[0,\beta_c(d))$. I.e., the critical Curie-Weiss perturbation for \textbf{any} $\beta<\beta_c(d)$ yields essentially the same non-Gaussian limit as occurs for $\beta=0$.
\end{remark}

\begin{remark}
	In combination with Theorem 1(b) of \cite{LR72}, Corollary \ref{cor:zero-free} implies that as long as $\beta\in[0,\beta_c(d))$, the partition function of the anti-ferromagnetic Ising model on $\mathbb{Z}^d$ has a fixed zero-free neighborhood of the origin which is independent of the domain where the model is defined.
\end{remark}

While Theorem \ref{thm:1LY} shows that it is the first Lee-Yang zero which is responsible for the closest singularity of the free energy, the thermodynamic behavior of other zeros is still open. For example, it is conjectured in \cite{LY52} that in the thermodynamic limit, the distribution of all Lee-Yang zeros has a density from which one should be able to extract information about critical phenomena, but the existence of such a density is proved in~\cite{BBCKK04} only for very large~$\beta$.

In the next section, we prove Theorem \ref{thm:1LY}. The proof makes use of three key ingredients: the monotonicity of Ursell functions (and thus monotonicity of the first Lee-Yang zero) from \cite{CJN22}, bounds on the derivatives of various orders of the free energy from \cite{Leb72} and their consequences (see Theorem \ref{thm:Leb} below), and the analyticity of the free energy for $\beta\in[0,\beta_c(d))$ from~\cite{Ott20}.

\section{Proof of the the main result}
Let $M_{\Lambda,\beta,h}$ be the \textbf{total magnetization} of the Ising model defined by \eqref{eq:Isingdef}:
\begin{equation}
	M_{\Lambda,\beta,h}:=\sum_{v\in\Lambda}\sigma_v.
\end{equation}
The \textbf{cumulants} of $M_{\Lambda,\beta,h}$, $u_k(M_{\Lambda,\beta,h})$, are defined by
\begin{equation}\label{eq:u_kMdef}
	u_k(M_{\Lambda,\beta,h}):=\left.\frac{d^k}{d t^k}\ln\left\langle\exp[t M_{\Lambda,\beta,h}]\right\rangle_{\Lambda,\beta,h}\right|_{t=0}, \forall k\in\mathbb{N},
\end{equation}
where $\langle\cdot\rangle_{\Lambda,\beta,h}$ denotes the expectation with respect to $\mathbb{P}_{\Lambda,\beta,h}$. Then it is easy to see that 
\begin{equation}
	u_k(M_{\Lambda,\beta,h})=\frac{d^k\ln Z_{\Lambda,\beta,h}}{d h^k}, \forall k\in\mathbb{N}.
\end{equation}
So we have
\begin{equation}\label{eq:f_VTaylor}
	f_{\Lambda,\beta}(h)=f_{\Lambda,\beta}(0)+\sum_{k=1}^{\infty}\frac{u_k(M_{\Lambda,\beta,0})}{|\Lambda|}\frac{h^k}{k!}, \forall h\in\mathbb{C} \text{ with }|h|<\alpha_1(\Lambda,\beta),
\end{equation}
where we have used the fact that $\ln Z_{\Lambda,\beta,h}$ is analytic in $\{h\in\mathbb{C}:|h|<\alpha_1(\Lambda,\beta)\}$.

For any $k\in\mathbb{N}$ and $v_1,\dots,v_k\in \Lambda$, the \textbf{Ursell function} $u_k^{\Lambda,\beta,h}(\sigma_{v_1},\dots,\sigma_{v_k})$ for $\{\sigma_v:v\in\Lambda\}$ from $\mathbb{P}_{\Lambda,\beta,h}$ is defined by
\begin{equation}\label{eq:u_kdef}
	u_k^{\Lambda,\beta,h}(\sigma_{v_1},\dots,\sigma_{v_k}):=\left.\frac{\partial^k}{\partial t_1\dots\partial t_k}\ln\left\langle\exp\left[\sum_{j=1}^k t_j\sigma_{v_j}\right]\right\rangle_{\Lambda,\beta,h}\right|_{t_1=\dots=t_k=0}.
\end{equation}
Comparing \eqref{eq:u_kMdef} and \eqref{eq:u_kdef}, we have
\begin{equation}\label{eq:u_kMsum}
	u_k(M_{\Lambda,\beta,h})=\sum_{v_1,\dots,v_k\in \Lambda}u_k^{\Lambda,\beta,h}(\sigma_{v_1},\dots,\sigma_{v_k}).
\end{equation}
It is known (and can be shown by the GKS inequalities \cite{Gri67,KS68}) that for any $\beta\geq0$, $h\in\mathbb{R}$, $\mathbb{P}_{\Lambda,\beta,h}$ converges weakly as $\Lambda\uparrow\mathbb{Z}^d$ to an infinite-volume measure whose expectation we denote by $\langle\cdot\rangle_{\mathbb{Z}^d,\beta,h}$, and this limit measure is translation invariant. So we can define
\begin{equation}
	u_k^{\mathbb{Z}^d,\beta,h}(\sigma_{v_1},\dots,\sigma_{v_k}):=\left.\frac{\partial^k}{\partial t_1\dots\partial t_k}\ln\left\langle\exp\left[\sum_{j=1}^k t_j\sigma_{v_j}\right]\right\rangle_{\mathbb{Z}^d,\beta,h}\right|_{t_1=\dots=t_k=0};
\end{equation}
the convergence of $\mathbb{P}_{\Lambda,\beta,h}$ implies that
\begin{equation}
	\lim_{\Lambda\uparrow\mathbb{Z}^d}u_k^{\Lambda,\beta,h}(\sigma_{v_1},\dots,\sigma_{v_k})=u_k^{\mathbb{Z}^d,\beta,h}(\sigma_{v_1},\dots,\sigma_{v_k}).
\end{equation}

We first prove that the coefficients of the power series in \eqref{eq:f_VTaylor} converge as $\Lambda\uparrow\mathbb{Z}^d$. Recall that $B_n=[-n,n]^d\cap\mathbb{Z}^d$.
\begin{proposition}\label{prop:b_kdef}
	For any $d\in\mathbb{N}$ and $\beta\geq 0$, we have
	\begin{equation}\label{eq:b_kdef}
		\lim_{n\rightarrow\infty}\frac{u_k(M_{B_n,\beta,0})}{|B_n|}=b_k(\beta):=\sum_{v_1,\dots,v_{k-1}\in\mathbb{Z}^d}	u_k^{\mathbb{Z}^d,\beta,0}(\sigma_0,\sigma_{v_1},\dots,\sigma_{v_{k-1}}), \forall k\in\mathbb{N},
	\end{equation}
where the subscript $0$ in $\sigma_0$ always refers to the origin in $\mathbb{Z}^d$, and the sum over an empty set is $0$, so $b_1(\beta)=0$. 
\end{proposition}
\begin{proof}
	It is enough to prove the proposition for even $k$ since both sides of \eqref{eq:b_kdef} are $0$ if $k$ is odd. By the signs and mononoticity of Ursell functions (see \cite{Shl86} and Theorem 1 of \cite{CJN22} respectively), we have
	\begin{align}
		&(-1)^{k-1}u_{2k}^{\Lambda,\beta,h}(\sigma_{v_1},\dots,\sigma_{v_{2k}})\geq 0, \forall v_1,\dots, v_{2k}\in \Lambda,\\
		&(-1)^{k-1}u_{2k}^{\Lambda_1,\beta,h}(\sigma_{v_1},\dots,\sigma_{v_{2k}})\leq (-1)^{k-1}u_{2k}^{\Lambda_2,\beta,h}(\sigma_{v_1},\dots,\sigma_{v_{2k}}), \forall \Lambda_1\subset\Lambda_2, \forall v_1,\dots, v_{2k}\in \Lambda_1.
	\end{align}
These, combined with \eqref{eq:u_kMsum} and translation invariance of $\langle\cdot\rangle_{\mathbb{Z}^d,\beta,h}$, imply that
\begin{align}\label{eq:u_2kupperbd}
	\frac{|u_{2k}(M_{B_n},\beta,0)|}{|B_n|}&=\frac{\sum_{u,v_1,\dots,v_{2k-1}\in B_n}|u_{2k}^{B_n,\beta,0}(\sigma_u,\sigma_{v_1},\dots,\sigma_{v_{2k-1}})|}{|B_n|}\nonumber\\
	&\leq\frac{\sum_{u,v_1,\dots,v_{2k-1}\in B_n}|u_{2k}^{\mathbb{Z}^d,\beta,0}(\sigma_u,\sigma_{v_1},\dots,\sigma_{v_{2k-1}})|}{|B_n|}\nonumber\\
	&\leq\frac{\sum_{u\in B_n; v_1,\dots,v_{2k-1}\in\mathbb{Z}^d}|u_{2k}^{\mathbb{Z}^d,\beta,0}(\sigma_u,\sigma_{v_1},\dots,\sigma_{v_{2k-1}})|}{|B_n|}\nonumber\\
	&= \sum_{v_1,\dots,v_{2k-1}\in \mathbb{Z}^d}|u_{2k}^{\mathbb{Z}^d,\beta,0}(\sigma_0,\sigma_{v_1},\dots,\sigma_{v_{2k-1}})|.
\end{align}	
Similarly, we also have that for any $\lambda\in(0,1)$, and then for large $n$,
\begin{align}
	&\frac{\sum_{u,v_1,\dots,v_{2k-1}\in B_n}|u_{2k}^{B_n,\beta,0}(\sigma_u,\sigma_{v_1},\dots,\sigma_{v_{2k-1}})|}{|B_n|}\nonumber\\
	&\quad \geq\frac{|B_{\lambda n}|\sum_{v_1,\dots,v_{2k-1}\in B_{(1-\lambda)n}}|u_{2k}^{B_{(1-\lambda)n},\beta,0}(\sigma_0,\sigma_{v_1},\dots,\sigma_{v_{2k-1}})|}{|B_n|}\nonumber\\
	&\rightarrow \lambda^d\sum_{v_1,\dots,v_{2k-1}\in \mathbb{Z}^d}|u_{2k}^{\mathbb{Z}^d,\beta,0}(\sigma_0,\sigma_{v_1},\dots,\sigma_{v_{2k-1}})| \text{ as }n\rightarrow\infty.
\end{align}
This, combined with \eqref{eq:u_2kupperbd}, completes the proof of the proposition by letting $\lambda\uparrow 1$.
\end{proof}
\begin{remark}\label{rem:b_k}
	 For $\beta\in[0,\beta_c(d))$, by the exponential decay of truncated two-point functions (see \cite{ABF87}) and the bounds on Ursell functions in terms of truncated two-point functions (see (1.13) of \cite{Leb72}), we know that the RHS of \eqref{eq:b_kdef} is finite in this case. But the RHS of \eqref{eq:b_kdef} might be $\infty$ for $\beta\geq\beta_c(d)$. For example, Simon's inequality \cite{Sim80} implies that
	 \begin{equation}
	 	\sum_{v\in\partial B(n)}\langle\sigma_0\sigma_v\rangle_{\mathbb{Z}^d,\beta_c(d),0}\geq1, \forall n\in \mathbb{N},
	 \end{equation}
 where $\partial B(n)$ is the set of vertices in $B(n)$ which have a nearest-neighbor in $\mathbb{Z}^d\setminus B(n)$. Therefore,
 \begin{equation}
 	\sum_{v\in\mathbb{Z}^d}\langle\sigma_0\sigma_v\rangle_{\mathbb{Z}^d,\beta_c(d),0}=\infty,
 \end{equation}
and by the GKS inequalities, we have
\begin{equation}\label{eq:susdiv}
	\lim_{n\rightarrow\infty}\frac{u_2(M_{B_n,\beta,0})}{|B_n|}=\sum_{v\in\mathbb{Z}^d}\langle\sigma_0\sigma_v\rangle_{\mathbb{Z}^d,\beta,0}=\infty, \forall \beta\geq\beta_c(d).
\end{equation}
\end{remark}

By Remark \ref{rem:b_k}, we know $b_k$ is finite for each $k\in\mathbb{N}$ when $\beta\in[0,\beta_c(d))$. We define
\begin{equation}\label{eq:rdef}
	r(\beta):=\begin{cases}
		1/\limsup_{k\rightarrow\infty}\left[|b_k(\beta)|/k!\right]^{1/k}, &\beta\in[0,\beta_c(d))\\
		0,&\beta\geq\beta_c(d),
	\end{cases}
\end{equation}
where we use the convention $1/0=\infty, 1/\infty=0$. We recall that $\alpha_1(\Lambda,\beta)$ is the modulus of the first zero of $Z_{\Lambda,\beta,h}$. We next prove
\begin{proposition}\label{prop:alpha1limit}
	For any $d\in\mathbb{N}$ and $\beta\geq0$, we have
	\begin{equation}
		\lim_{n\rightarrow\infty}\alpha_1(B_n,\beta)=r(\beta),
	\end{equation}
where $r(\beta)$ is defined in \eqref{eq:rdef}.
\end{proposition}
\begin{remark}
	This proposition should be contrasted to a similar result for the antiferromagnetic Ising model due to Lebowitz; see (4.3) of \cite{Pen63} or Chapter 6 of \cite{Mcc15}. 
\end{remark}
\begin{proof}[Proof of Proposition \ref{prop:alpha1limit}]
	The Hadamard factorization theorem and the fact that $Z_{\Lambda,\beta,h}$ is a function of $h$  of exponential order $1$ imply that
	\begin{equation}
		Z_{\Lambda,\beta,h}=Z_{\Lambda,\beta,0}\prod_{j=1}^{\infty}\left(1+\frac{h^2}{\alpha_j^2(\Lambda,\beta)}\right), \forall h\in\mathbb{C},
	\end{equation}
where $0<\alpha_1(\Lambda,\beta)\leq\alpha_2(\Lambda,\beta)\leq\dots$ are all the positive zeros of $Z_{\Lambda,\beta,ih}$ as a function of $h$ (listed according to their multiplicities), and $\sum_{j=1}^{\infty}\alpha_j^{-2}(\Lambda,\beta)<\infty$; see Lemma~1 of~\cite{HJN22}.
Combining this with \eqref{eq:f_Vdef} and \eqref{eq:f_VTaylor}, we have
	\begin{equation}
		\sum_{k=1}^{\infty}\frac{u_k(M_{\Lambda,\beta,0})}{|\Lambda|}\frac{h^k}{k!}=\ln\prod_{j=1}^{\infty}\left(1+\frac{h^2}{\alpha_j^2(\Lambda,\beta)}\right), \forall h\in\mathbb{C} \text{ with }|h|<\alpha_1(\Lambda,\beta).
	\end{equation}
Using the Taylor series for $\ln(1+x)$, we get
\begin{equation}\label{eq:u_kMandalpha}
	u_{2k-1}(M_{\Lambda,\beta,0})=0, u_{2k}(M_{\Lambda,\beta,0})=(-1)^{k-1}\frac{(2k)!}{k}\sum_{j=1}^{\infty}\frac{1}{\alpha_j^{2k}(\Lambda,\beta)}, \forall k\in\mathbb{N}.
\end{equation}
Note that, as a function of $h$, $Z_{\Lambda,\beta,ih}$ has a period $2\pi$ and it has exactly $2|\Lambda|$ zeros in $[0,2\pi)$; see the discussion before Theorem 1 of \cite{HJN22} for more details. So we have (to simplify notation, we write $\alpha_j$ for $\alpha_j(\Lambda,\beta)$ )
\begin{align}\label{eq:sumalphsub}
	\sum_{j=1}^{\infty}\frac{1}{\alpha_j^{2k}}=\sum_{l=0}^{\infty}\sum_{m=1}^{2|\Lambda|}\frac{1}{(\alpha_m+2l\pi)^{2k}}\leq\sum_{l=0}^{\infty}\frac{2|\Lambda|}{(\alpha_1+2l\pi)^{2k}}\leq 4|\Lambda|\alpha_1^{-2k}, \forall k\in\mathbb{N}.
\end{align}
The last inequality in \eqref{eq:sumalphsub} follows, after dividing by $\alpha_1^{-2k}$, by using that
\begin{equation}
	\alpha_1<2\pi \text{ and } \sum_{l=0}^{\infty}\frac{1}{(1+l)^2}=\frac{\pi^2}{6}<2.
\end{equation}
This, combined with \eqref{eq:u_kMandalpha}, implies that
\begin{equation}\label{eq:u_2kupperbd1}
	|u_{2k}(M_{\Lambda,\beta,0})|\leq\frac{(2k)!}{k}4|\Lambda|[\alpha_1(\Lambda,\beta)]^{-2k}, \forall k\in\mathbb{N}.
\end{equation}
Setting $\Lambda=B_n$ and using \eqref{eq:susdiv}, we have
\begin{equation}
	\infty=\liminf_{n\rightarrow\infty}\frac{u_2(M_{B_n,\beta,0})}{|B_n|}\leq 8\liminf_{n\rightarrow\infty}[\alpha_1(B_n,\beta)]^{-2}, \forall \beta\geq\beta_c(d),
\end{equation}
which of course implies that
\begin{equation}\label{eq:alphalimitlarge}
	\lim_{n\rightarrow\infty}\alpha_1(B_n,\beta)=0, \forall \beta\geq\beta_c(d).
\end{equation}
Since the power series in \eqref{eq:f_VTaylor} has radius of convergence $\alpha_1(\Lambda,\beta)$, we have
\begin{equation}
	\frac{1}{\alpha_1(\Lambda,\beta)}=\limsup_{k\rightarrow\infty}\left[\frac{|u_k(M_{\Lambda,\beta,0})|}{|\Lambda|}\frac{1}{k!}\right]^{1/k}=\limsup_{k\rightarrow\infty}\left[\frac{|u_{2k}(M_{\Lambda,\beta,0})|}{|\Lambda|}\frac{1}{(2k)!}\right]^{1/(2k)},
\end{equation}
where we have used the first equation in \eqref{eq:u_kMandalpha} in the last equality. Combining this with \eqref{eq:u_2kupperbd}, \eqref{eq:b_kdef} and \eqref{eq:rdef} (note that $b_k(\beta)=0$ for odd $k$), we get
\begin{equation}\label{eq:alpha_1lower}
	\frac{1}{\alpha_1(B_n,\beta)}\leq\limsup_{k\rightarrow\infty}\left[\frac{|b_{2k}(\beta)|}{(2k)!}\right]^{1/(2k)}=\frac{1}{r(\beta)}, \forall \beta\in[0,\beta_c(d)).
\end{equation}
Proposition \ref{prop:b_kdef} and \eqref{eq:u_2kupperbd1} imply that
\begin{align}
	|b_{2k}(\beta)|&=\liminf_{n\rightarrow\infty}\frac{|u_{2k}(M_{B_n,\beta,0})|}{|B_n|}\leq\liminf_{n\rightarrow\infty}\frac{4(2k)!}{k}[\alpha_1(B_n,\beta)]^{-2k}\nonumber\\
	&=\frac{4(2k)!}{k}[\limsup_{n\rightarrow\infty}\alpha_1(B_n,\beta)]^{-2k}.
\end{align}
Hence,
\begin{align}
	\frac{1}{r(\beta)}&=\limsup_{k\rightarrow\infty}\left[\frac{|b_{2k}(\beta)|}{(2k)!}\right]^{1/(2k)}\leq\limsup_{k\rightarrow\infty}\left[\frac{4}{k}\right]^{1/(2k)}[\limsup_{n\rightarrow\infty}\alpha_1(B_n,\beta)]^{-1}\nonumber\\
	&=[\limsup_{n\rightarrow\infty}\alpha_1(B_n,\beta)]^{-1}.
\end{align}
Combining this with \eqref{eq:alpha_1lower}, we obtain
\begin{equation}
	\lim_{n\rightarrow\infty}\alpha_1(B_n,\beta)=r(\beta), \forall \beta\in[0,\beta_c(d)).
\end{equation}
This and \eqref{eq:alphalimitlarge} complete the proof of the proposition.
\end{proof}
The following theorem says that for $\beta\in[0,\beta_c(d))$, the free energy $f_{\beta}$ is infinitely differentiable in $h$ and the derivatives of $f_{B_n,\beta}$ converge to those of $f_{\beta}$.
\begin{theorem}[\cite{Leb72} and \cite{ABF87}]\label{thm:Leb}
	For any $d\in\mathbb{N}$ and $\beta\in[0,\beta_c(d))$, $f_{\beta}$ is infinitely differentiable in $h\in\mathbb{R}$. Moreover, for each $k\in\mathbb{N}$,
	\begin{equation}\label{eq:f_Vdiffconv}
		\lim_{n\rightarrow\infty}\frac{d^k f_{B_n,\beta}}{d h^k}(h)=\frac{d^k f_{\beta}}{d h^k}(h) \text{ uniformly in any compact subset of }\mathbb{R}.
	\end{equation}
\end{theorem}
\begin{proof}
	It is proved in \cite{Leb72} that $f_{\beta}$ is infinitely differentiable in $h$ whenever the infinite-volume truncated two-point functions decay exponentially; the latter is proved in \cite{ABF87} (see also \cite{DCT16} for an alternative proof). See also Section II.12 of \cite{Sim93} for more details; in particular, see Lemma II.12.9 of \cite{Sim93} for the argument which leads to \eqref{eq:f_Vdiffconv} by using the bounds on the derivatives of $f_{B_n,\beta}$ from \cite{Leb72}.
\end{proof}

We next relate $r(\beta)$ to the free energy $f_{\beta}$ for $\beta\in[0,\beta_c(d))$. 
\begin{proposition}\label{prop:fTaylor}
	For any $d\in\mathbb{N}$ and any $\beta\in[0,\beta_c(d))$, we have
	\begin{equation}\label{eq:fTaylor}
		f_{\beta}(h)=f_{\beta}(0)+\sum_{k=1}^{\infty}b_k(\beta)\frac{h^k}{k!}, \forall h\in\mathbb{C} \text{ with }|h|<r(\beta),
	\end{equation}
where $b_k(\beta)$ are defined in \eqref{eq:b_kdef} and $r(\beta)$ is defined in \eqref{eq:rdef}. Moreover, $r(\beta)>0$ is the radius of the largest disk centered at the origin where $f_{\beta}$ is analytic.
\end{proposition}
\begin{proof}
	Corollary 1.4 of \cite{Ott20} says that $f_{\beta}$ is analytic in a neighborhood of $h=0$. As a result, we have
	\begin{equation}\label{eq:fTaylor1}
		f_{\beta}(h)=f_{\beta}(0)+\sum_{k=1}^{\infty}\frac{f_{\beta}^{(k)}(0)}{k!}h^k, \forall h \text{ in some neighborhood of }h=0.
	\end{equation}
Theorem \ref{thm:Leb}, \eqref{eq:f_VTaylor} and Proposition \ref{prop:b_kdef} give
\begin{equation}\label{eq:fTaylorcoeff}
	f_{\beta}^{(k)}(0)=\lim_{n\rightarrow\infty}f_{B_n,\beta}^{(k)}(0)=\lim_{n\rightarrow\infty}\frac{u_k(M_{B_n,\beta,0})}{|B_n|}=b_k(\beta), \forall k\in\mathbb{N}.
\end{equation}	
This completes the proof of \eqref{eq:fTaylor}. The second part of the proposition is clear from \eqref{eq:fTaylor1}, \eqref{eq:fTaylorcoeff} and \eqref{eq:rdef}.
\end{proof}

It is known that the free energy $f_{\beta}$ is not differentiable at $h=0$ if $\beta\in(\beta_c(d),\infty)$; see, e.g., Theorems 3.25 and 3.34 of \cite{FV18} for a proof. It is also known that $f_{\beta_c(d)}$ is differentiable at $h=0$ ; see, e.g., Theorem 3.34 of \cite{FV18} together with \cite{Yan52} for $d=2$, \cite{AF86} for $d\geq 4$, and \cite{ADCS15} for $d=3$. It is expected that $f_{\beta_c(d)}$ is not twice differentiable at $h=0$ because the susceptibility diverges at $\beta_c(d)$. But we didn't find a proof in the literature. So for completeness, we include a proof here.
\begin{lemma}\label{lem:fnottwice}
	$f_{\beta_c(d)}$  is not twice differentiable at $h=0$. For each $\beta\geq\beta_c(d)$, $f_{\beta}$ is not analytic in any neighborhood of $h=0$.
\end{lemma}
\begin{proof}
	We recalled right before the lemma that $f_{\beta_c(d)}$ is differentiable at $h=0$. Theorem~3.34 and Proposition 3.29 of \cite{FV18} imply that
	\begin{equation}\label{eq:fdiff}
		f_{\beta_c(d)}^{\prime}(h)=\langle\sigma_0\rangle_{\mathbb{Z}^d,\beta_c(d),h}, \forall h\in\mathbb{R}.
	\end{equation}
Here, we have used the fact that the Gibbs state is unique at $\beta=\beta_c(d)$ and any $h\in\mathbb{R}$. The GKS inequalities imply that
\begin{equation}
	\langle\sigma_0\rangle_{B_n,\beta_c(d),h}\uparrow\langle\sigma_0\rangle_{\mathbb{Z}^d,\beta_c(d),h} \text{ as }n\uparrow\infty, \forall h\geq 0.
\end{equation}
For any fixed $h>0$, by the mean value theorem, there exists $h_0\in(0,h)$ such that
\begin{align}
	\frac{\langle\sigma_0\rangle_{\mathbb{Z}^d,\beta_c(d),h}-\langle\sigma_0\rangle_{\mathbb{Z}^d,\beta_c(d),0}}{h}&=\frac{\langle\sigma_0\rangle_{\mathbb{Z}^d,\beta_c(d),h}}{h}\geq\frac{\langle\sigma_0\rangle_{B_n,\beta_c(d),h}}{h}=\left.\frac{d\langle\sigma_0\rangle_{B_n,\beta_c(d),h}}{dh}\right|_{h=h_0}\nonumber\\
	&=\sum_{v\in B_n}\left[\langle\sigma_0\sigma_v\rangle_{B_n,\beta_c(d),h_0}-\langle\sigma_0\rangle_{B_n,\beta_c(d),h_0}\langle\sigma_v\rangle_{B_n,\beta_c(d),h_0}\right],
\end{align}
where we have used $\langle\sigma_0\rangle_{\mathbb{Z}^d,\beta_c(d),0}=\langle\sigma_0\rangle_{B_n,\beta_c(d),0}=0$.
Taking lim inf, we obtain
\begin{equation}
	\liminf_{h\downarrow0}\frac{\langle\sigma_0\rangle_{\mathbb{Z}^d,\beta_c(d),h}-\langle\sigma_0\rangle_{\mathbb{Z}^d,\beta_c(d),0}}{h}\geq \sum_{v\in B_n}\langle\sigma_0\sigma_v\rangle_{B_n,\beta_c(d),0}.
\end{equation}
Letting $n\rightarrow\infty$ and applying \eqref{eq:susdiv}, we have
\begin{equation}
	\liminf_{h\downarrow0}\frac{\langle\sigma_0\rangle_{\mathbb{Z}^d,\beta_c(d),h}-\langle\sigma_0\rangle_{\mathbb{Z}^d,\beta_c(d),0}}{h}=\infty.
\end{equation}
This and \eqref{eq:fdiff} complete the proof of the first part of the lemma. The second part follows from the first part and the discussion before the lemma. 
\end{proof}
We are ready to prove Theorem \ref{thm:1LY}.
\begin{proof}[Proof of Theorem \ref{thm:1LY}]
	We first note that by Corollary 1 of \cite{CJN22}, for fixed $\beta$, $\alpha_1(\Lambda_n,\beta)$ is decreasing in $n$. If $\Lambda_n=B_n$ for each $n\in\mathbb{N}$, then Theorem \ref{thm:1LY} follows from Propositions~\ref{prop:alpha1limit} and \ref{prop:fTaylor} and Lemma \ref{lem:fnottwice}.  For general $\Lambda_n$, we can find subsequences $\{\Lambda_{n_j}:j\in\mathbb{N}\}$ and $\{B_{n_j}:j\in\mathbb{N}\}$ such that
	\begin{equation}
		\Lambda_{n_1}\subset B_{n_1}\subset \Lambda_{n_2}\subset B_{n_2}\subset \Lambda_{n_3}\subset B_{n_3}\subset\dots.
	\end{equation}
By the monotonicty of $\alpha_1(\Lambda_{n_j},\beta)$ in $j$ and Theorem \ref{thm:1LY} for the the case $\Lambda_n=B_n$, we have
\begin{equation}
	\alpha_1(\Lambda_{n_j},\beta)\downarrow \alpha_1(\mathbb{Z}^d,\beta) \text{ as }j\uparrow\infty.
\end{equation}
The proof for the original sequence $\{\alpha_1(\Lambda_n,\beta):n\in\mathbb{N}\}$ follows again by the monotonicity of $\alpha_1(\Lambda_n,\beta)$ in $n$.
\end{proof}

\section*{Acknowledgments}
The research of the first author was partially supported by NSFC grant 11901394. The authors thank Federico Camia and Elliott Lieb for useful discussions.


\bibliographystyle{abbrv}
\bibliography{reference}

\end{document}